\documentclass[letterpaper, 10 pt, conference]{ieeeconf}  

\IEEEoverridecommandlockouts                              

\usepackage{graphics} 
\usepackage{epsfig} 
\usepackage{algorithm}
\usepackage{algorithmic}
\usepackage{amsmath} 
\usepackage{amssymb}  
\usepackage{xcolor}
\usepackage{bbm}
\graphicspath{{../figs/}}
\newtheorem{assumption}{Assumption}
\newtheorem{property}{Property}
\newtheorem{lemma}{Lemma}

\newtheorem{theorem}{Theorem}
\newtheorem{proposition}{Proposition}

\newtheorem{remark}{Remark}

\newcommand{\st}{\mbox{s.t.}}

\newcommand{\be}{\begin{equation}}
\newcommand{\ee}{\end{equation}}
\newcommand{\bea}{\begin{eqnarray}}
\newcommand{\eea}{\end{eqnarray}}

\newcommand{\bvec}{\left(\begin{array}{c}}
\newcommand{\evec}{\end{array}\right)}
\newcommand{\bsub}{\begin{subequations}}
\newcommand{\esub}{\end{subequations}}

\newcommand{\cm}{\color{black}}
\newcommand{\cmred}{\color{black}}

\begin{document}

\title{\LARGE\bf A Parallel Decomposition Scheme for Solving \\ Long-Horizon Optimal Control Problems}

\author{Sungho Shin$^{1}$, Timm Faulwasser$^{2}$, Mario Zanon$^{3}$, and Victor M. Zavala$^{1}$
  \thanks{$^{1}$S. Shin and V. M. Zavala are with the Department of Chemical and Biological Engineering, University of Wisconsin-Madison, Madison, WI 53706 USA (e-mail: sungho.shin@wisc.edu; victor.zavala@wisc.edu).}
  \thanks{$^{2}$Timm Faulwasser is with Institute for Automation and Applied Informatics, Karlsruhe Institute of Technology, 76021 Karlsruhe, Germany. (e-mail: timm.faulwasser@ieee.org)}
  \thanks{$^{3}$Mario Zanon is with IMT Lucca, 55100 Lucca, Italy.(e-mail: mario.zanon@imtlucca.it)}
  \thanks{\cmred{TF acknowledges support by the Bundesministerium f\"u{}r Bildung und Forschung (BMBF), Grant 05M18CKA.}}}

\maketitle

\begin{abstract}
  We present a temporal decomposition scheme for solving long-horizon optimal control problems. In the proposed scheme, the time domain is decomposed into a set of subdomains with partially overlapping regions. Subproblems associated with the subdomains are solved in parallel to obtain local primal-dual trajectories that are assembled to obtain the global trajectories.
  We provide a sufficient condition that guarantees convergence of the proposed scheme. This condition states that the effect of perturbations on the boundary conditions (i.e., the initial state and terminal dual/adjoint variable) should decay asymptotically as one moves away from the boundaries. 
  This condition also reveals that the scheme converges if the size of the overlap is sufficiently large and that the convergence rate improves with the size of the overlap. We prove that linear quadratic problems satisfy the asymptotic decay condition, and we discuss numerical strategies to determine if the condition holds in more general cases. We draw upon a non-convex optimal control problem to illustrate the performance of the proposed scheme.
\end{abstract}

\section{Introduction}\label{sec:intro}

Long-horizon optimal control problems (OCPs) arise in model predictive control (MPC) applications such as chemical process systems \cite{baldea2007control}, autonomous vehicle steering \cite{falcone2007predictive}, and battery systems \cite{kumar2018hierarchical}. They also appear in other application domains such as chemical production planning \cite{jackson2003temporal} and electricity production planning \cite{barrows2014time}. Different decomposition techniques have been reported in the literature to improve computational tractability of these problems including dual decomposition \cite{giselsson2013accelerated}, alternating direction method of multipliers \cite{boyd2011distributed}, dual dynamic programming \cite{geoffrion1972generalized}, Gauss-Seidel schemes \cite{zavala2016new,shin2018multi}, and parallel Newton schemes \cite{deng2018parallel}. {\cm Such decomposition techniques allow scalable solutions of long-horizon OCPs by the use of parallel computers.}

In this work, we study the convergence properties of a new {\em decomposition} paradigm that uses overlapping time domains. This approach is motivated by overlapping Schwartz schemes used for the solution of partial differential equations (PDEs) \cite{mathew2008domain,dryja1987additive}. In the proposed scheme, the time domain is decomposed into a set of partially overlapping subdomains. Subproblems associated with subdomains are solved in parallel to obtain local primal-dual trajectories by using current primal-dual information from the neighboring subdomains. The subdomain trajectories are then assembled (this can be interpreted as a projection operator) to update the primal-dual trajectories over the entire domain, and the procedure is repeated.  We provide a sufficient condition that guarantees convergence of the proposed decomposition algorithm when applied to OCPs. This condition applies to linear and nonlinear problems. Specifically, the condition indicates that convergence of the decomposition scheme is guaranteed provided that the sensitivity of the primal-dual (state-adjoint) trajectories to perturbations in initial states and terminal cost gradients decay asymptotically as one moves away from the boundaries. We call this condition {\em asymptotic decay of sensitivity} (ADS). The condition also reveals that the algorithm converges provided that the size of the overlap is sufficiently large and that the convergence rate improves as the size of the overlapping regions increases. 

{ADS-like properties have been recently explored in the literature. Xu et. al. recently showed that an ADS condition (in the primal space) holds under a time-varying and inequality-constrained linear quadratic (LQ) control setting \cite{xu2018exponentially,xu2017exponentially}. To prove this, the authors assumed uniformly complete controllability and exploited the algebraic structure of the Riccati equation. The authors used the primal ADS property to show that trajectories of an overlapping temporal decomposition scheme approximate those of the long-horizon problem and that the approximation error decays as the size of the overlap increases. The authors also showed that receding horizon control provides approximate trajectories and that the error converges as the size of the overlapping regions increase. These works do not provide an algorithmic scheme that delivers optimal trajectories (for a given size of the overlapping region), as we do in this work. Shin et. al. recently established a primal ADS condition for general graph-structured, unconstrained quadratic programs and showed that this condition guarantees the convergence of an iterative overlapping decomposition algorithm \cite{shin2018decentralized}.   The condition established in this work is specialized to constrained OCPs and operates in the primal-dual space. A primal ADS property has also been established in \cite[Lemma 5]{kit:faulwasser18e_2}. This condition is exploited by the authors to establish the stability of economic MPC.}

The paper is organized as follows. In Section \ref{sec:setting}, we present the basic setting and describe the proposed decomposition scheme. In Section \ref{sec:main}, we propose a primal-dual ADS condition that guarantees convergence of the algorithm. In Section \ref{sec:LQ}, we show that ADS holds for a simplified LQ setting. In Section \ref{sec:cstudy}, we demonstrate the proposed scheme using a nonlinear economic MPC problem. 

\section{Basic Definitions and Setting}\label{sec:setting}

We consider an OCP with a time domain set $\mathbb{I}_{M:N}$, an initial state $x^*_M\in\mathbb{R}^{n_x}$, and a terminal cost gradient $\lambda^*_N\in\mathbb{R}^{n_x}$ of the form:
\begin{subequations}
  \label{eqn:ocp}
  \begin{align}
    \label{eqn:ocp-obj}
    \min_{\substack{x_{M:N}\\u_{M:N-1}}}\;& \sum_{i=M}^{N-1} \ell(x_i,u_i)+ (\lambda^*_N)^\top x_{N}\\
    \label{eqn:ocp-init}
    \st \;&x_{M} = x^*_M \quad(\lambda_{M})\\
    \label{eqn:ocp-dyn}
    & x_{i} = f(x_{i-1},u_{i-1}) \quad(\lambda_{i})\quad \forall i\in\mathbb{I}_{M+1:N}\\
    \label{eqn:ocp-ineq}
    &g(x_i,u_i)\leq 0\quad (\mu_{i}),\quad \forall i\in\mathbb{I}_{M:N-1}
  \end{align}
\end{subequations}
We denote this problem as $\mathcal{P}_{M:N}(x^*_M,\lambda^*_N)$.  Here, $x_i\in\mathbb{R}^{n_x}$ and $u_i\in\mathbb{R}^{n_u}$ are the state and input variables at time $i$; $\lambda_{i}\in\mathbb{R}^{n_x}$ and $\mu_{i}\in\mathbb{R}^{n_g}$ are the dual variables associated with \eqref{eqn:ocp-init}-\eqref{eqn:ocp-dyn} and \eqref{eqn:ocp-ineq}, respectively. Symbol $f:\mathbb{R}^{n_x}\times\mathbb{R}^{n_u}\rightarrow\mathbb{R}^{n_x}$ is the dynamic mapping, $\ell:\mathbb{R}^{n_x}\times\mathbb{R}^{n_u}\rightarrow\mathbb{R}$ is the stage cost mapping, and $g:\mathbb{R}^{n_x}\times\mathbb{R}^{n_u}\rightarrow\mathbb{R}^{n_g}$ is the constraint mapping. We denote the set of real numbers and the set of integers as $\mathbb{R}$ and $\mathbb{Z}$, respectively, and we define $\mathbb{I}_{M:N}:=\mathbb{Z}\cap [M,N]$. We use the syntax $(x_1,x_2,\cdots,x_n) := \begin{bmatrix}x_1^\top & x_2^\top &\cdots &x_n^\top\end{bmatrix}^\top$ and $x_{M:N} := (x_{M},x_{M+1},\cdots,x_{N})$. We denote primal-dual pairs as $z_i:=(x_i,\lambda_i)$.

\begin{remark}\label{rmk:terminal}
  From the KKT conditions of \eqref{eqn:ocp}, it follows that $\lambda_N=\lambda^*_N$. Therefore, incorporating the terminal penalty term $(\lambda^*_N)^\top x_N$ in the objective function as \eqref{eqn:ocp-obj} essentially constitute the terminal constraint of the dual variable $\lambda$ at $i =N$. 
\end{remark}

We now describe the proposed decomposition scheme for solving $\mathcal{P}_{M:N}(x_M^*,\lambda_N^*)$. We partition the time domain  $\mathbb{I}_{M:N}$ into a collection of $K$ non-overlapping (i.e., disjoint) sets of the form $\mathbb{I}_{M_1:N_1},\cdots,\mathbb{I}_{M_K:N_K}$. We also define a collection of overlapping sets $\mathbb{I}_{M^\omega_1:N^\omega_1},\cdots,\mathbb{I}_{M^\omega_K:N^\omega_K}$ satisfying
\begin{align*}
  M_k^\omega = \max(M_k-\omega,M),\; N_k^\omega = \min(N_k+\omega,N)
\end{align*}
for any $k\in\mathbb{I}_{1:K}$. We call $\omega\in\mathbb{Z}_{>0}$ the {\em size of the overlap}.  The non-overlapping and overlapping time domains with $\omega=1$ are illustrated in Figure \ref{fig:cartoon}.

\begin{figure*}[t!]
\begin{center}
  \includegraphics[width=.8\textwidth]{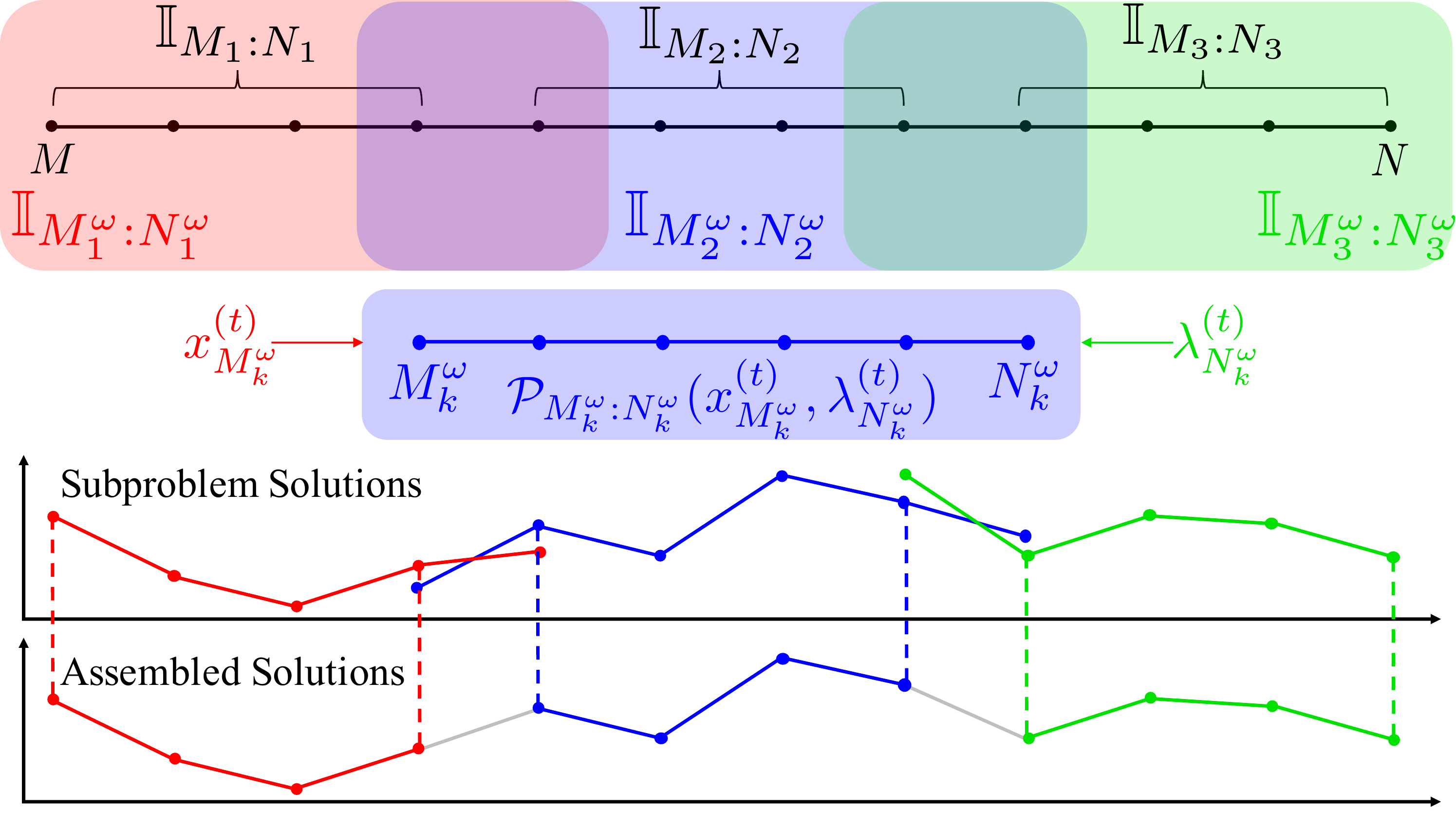}
  \caption{Sketch of parallel decomposition algorithm.}\label{fig:cartoon}
  \end{center}
\end{figure*}

The scheme starts with an initial guess of the primal $x^{(0)}_{M:N}$ and dual $\lambda^{(0)}_{M:N}$ trajectories with $x^{(0)}_M=x_M^*$ and $\lambda^{(0)}_N=\lambda_N^*$. For each subdomain index $k\in\mathbb{I}_{1:K}$ and iteration counter $t\in\mathbb{Z}_{>0}$, subproblems $\mathcal{P}_{M^\omega_k:N^\omega_k}(x^{(t)}_{M_k^\omega},\lambda^{(t)}_{N_k^\omega})$ are solved in parallel. Each subproblem requires current primal-dual information from the neighboring subproblems (this information enters in the initial condition and terminal penalty).
The solution of subproblem $\mathcal{P}_{M^\omega_k:N^\omega_k}(x^{(t)}_{M_k^\omega},\lambda^{(t)}_{N_k^\omega})$ yields the local primal-dual trajectories $\hat{z}^{(t+1,k)}_{M^\omega_k:N^\omega_k}$. These trajectories are restricted to the non-overlapping subdomains $\mathbb{I}_{M_k:N_k}$ by using the restriction $z^{(t+1)}_{M_k:N_k} := \hat{z}^{(t+1,k)}_{M_k:N_k}$. The solutions at $\mathbb{I}_{M^\omega_k:N^\omega_k}\setminus \mathbb{I}_{M_k:N_k}$ are discarded. This restriction procedure over $k\in\mathbb{I}_{1:K}$ can be seen as a projection that assembles the entire primal-dual trajectory $z^{(t+1)}_{M:N}$ that is in turn used as the next guess in the algorithm.  

The trajectory assembling procedure is sketched in Figure \ref{fig:cartoon}. The proposed decomposition scheme is summarized in Algorithm \ref{alg:main} and can be stated compactly as follows:
\begin{align}\label{eqn:alg}
  z^{(t+1)}_{M_k:N_k} \leftarrow \mathcal{P}_{M^\omega_k:N^\omega_k}(x^{(t)}_{M^\omega_k},\lambda^{(t)}_{N^\omega_k}),\; k\in\mathbb{I}_{1:K},\; t\in \mathbb{Z}_{>0}.
\end{align}
\begin{algorithm}[H]\caption{Decomposition Scheme with Overlap}\label{alg:main}
  \begin{algorithmic}
    \STATE Initialize $z^{(0)}_{M:N}$ and $t\leftarrow 0$
    \WHILE {termination criteria unsatisfied}
    \FOR {(in parallel) $k=1$ to $K$}
    \STATE Solve $\mathcal{P}_{M^\omega_k:N^\omega_k}(x^{(t)}_{M_k^\omega},\lambda^{(t)}_{N_k^\omega})$ to obtain $\hat{z}^{(t+1,k)}_{M^\omega_k:N^\omega_k}$
    \STATE Apply restriction $z^{(t+1)}_{M_k:N_k} := \hat{z}^{(t+1,k)}_{M_k:N_k}$
    \ENDFOR 
    \STATE Evaluate termination criteria
    \STATE $t\leftarrow t+1$
    \ENDWHILE
  \end{algorithmic} 
\end{algorithm}

\section{Convergence Results}\label{sec:main}

We now provide a sufficient condition for \eqref{eqn:ocp} that guarantees convergence of the decomposition algorithm \eqref{eqn:alg}.  We begin by making the following existence and uniqueness assumption on solutions of problem \eqref{eqn:ocp}.

\begin{assumption}[Existence and uniqueness of solution]\label{ass:uniqsol}
  There exist sets $X$ and $\Lambda \subseteq \mathbb{R}^{n_x}$ such that, for any $M,N\in\mathbb{Z}_{>0}$ with $M< N$ and $(x^*_M,\lambda^*_N)\in X\times\Lambda$, there exists a unique primal-dual solution $z^*_{M:N}$ of $\mathcal{P}_{M:N}(x^*_M,\lambda^*_N)$ with $z_i\in X\times\Lambda$ and for any $i\in\mathbb{I}_{M:N}$.
\end{assumption}
\begin{remark}
The solution of $\mathcal{P}_{M:N}(x^*_M,\lambda^*_N)$ always satisfies $x_M=x^*_M$ (by \eqref{eqn:ocp-init}) and $\lambda_N=\lambda^*_N$ (see Remark \ref{rmk:terminal}), and thus we can write the solution of $\mathcal{P}_{M:N}(x^*_M,\lambda^*_N)$ as $z^*_{M:N}=(x^*_{M:N},\lambda^*_{M:N})$. Because the scheme \eqref{eqn:alg} is mostly concerned with the state and the adjoints (this is the information exchanged between domains), we will not explicitly indicate the controls $u^*_{M:N-1}$ and multipliers $\mu^*_{M:N-1}$ in the nomenclature.
\end{remark}
 We now state the following principle of optimality result:
\begin{lemma}\label{lem:fixed}
  Assume that problem \eqref{eqn:ocp} satisfies Assumption \ref{ass:uniqsol} and consider $P,M,N,L\in\mathbb{Z}_{>0}$ with $P\leq M< N\leq L$ and $(x^*_P,\lambda^*_L)\in X\times\Lambda$. Let $z^*_{P:L}$ be the solution of $\mathcal{P}_{P:L}(x^*_P,\lambda^*_L)$, then $z^*_{M:N}$ is the solution of $\mathcal{P}_{M:N}(x^*_M,\lambda^*_N)$. 
\end{lemma}

\begin{proof}
  We define the Lagrangian $\mathcal{L}_{M:N}(\cdot;\cdot)$ of $\mathcal{P}_{M:N}(x^*_M,\lambda^*_N)$ as:
  \begin{align}\label{eqn:lag}
      &\mathcal{L}_{M:N}(x_{M:N},u_{M:N-1},\lambda_{M:N},\mu_{M:N-1};x^*_M,\lambda^*_N) \\
      &:=\sum_{i=M}^{N-1} \ell(x_i,u_i) + \lambda_M^\top(x_M-x_M^*) + (\lambda^*_N)^\top x_N\nonumber\\
      &+ \sum_{i=M+1}^{N} \lambda_{i}^\top (x_{i}-f(x_{i-1},u_{i-1}))+ \sum_{i=M}^{N-1} \mu^\top_i g(x_i,u_i)\nonumber
  \end{align}
  From Assumption \ref{ass:uniqsol}, the solution of $\mathcal{P}_{1:L}(x^*_1,\lambda^*_L)$ exists and is unique and this implies that the KKT conditions hold at $x^*_{M:N}$, $u^*_{M:N-1}$, $\lambda^*_{M:N}$, $\mu^*_{M:N-1}$; that is,
  \begin{subequations}\label{eqn:kkt}
    \begin{align}
      \label{eqn:kkt-duf-x}\lambda^*_i +  \left(\nabla_{x}f^*_i\right)^\top \lambda^*_{i+1} +\nabla_x \ell^*_i + \left(\nabla_{x}g^*_i\right)^\top \mu^*_{i}&= 0  \\
      \label{eqn:kkt-duf-u}\left(\nabla_{u}f^*_i\right)^\top \lambda^*_{i+1} + \nabla_{u} \ell^*_i
      +  \left(\nabla_{u}g^*_i\right)^\top \mu^*_{i}&= 0\\
      \label{eqn:kkt-prf}x^*_{i+1} &= f^*_i\\
      \label{eqn:kkt-slk}g^*_i\leq 0,\quad \mu^*_i \geq 0,\quad \mathop{\textrm{diag}}(\mu^*_i) g^*_i &= 0
  \end{align}
  \end{subequations}
  holds for any $i\in\mathbb{I}_{M:N-1}$, where $ f^*_i:= f(x^*_i,u^*_i)$, $\nabla f^*_i:=\nabla f(x^*_i,u^*_i)$, $\nabla g^*_i:=\nabla g(x^*_i,u^*_i)$, $\nabla \ell^*_i:=\nabla \ell(x^*_i,u^*_i)$.
  
  From \eqref{eqn:kkt}, we observe that the KKT conditions of $\mathcal{P}_{P:L}(x^*_1,\lambda^*_L)$ with $(x^*_{P:L},u^*_{P:L-1},\lambda^*_{P:L},\mu^*_{P:L-1})$ imply that the KKT conditions of $\mathcal{P}_{M:N}(x^*_M,\lambda^*_N)$ hold with $(x^*_{M:N},u^*_{M:N-1},\lambda^*_{M:N},\mu^*_{M:N-1})$. This consistency is the result of using the duals as a terminal penalty in \eqref{eqn:ocp}. 
  
By Assumption \ref{ass:uniqsol}, we have that $(x^*_M,\lambda^*_N)\in X\times \Lambda$. This implies that the solution of $\mathcal{P}_{M:N}(x^*_M,\lambda^*_N)$ exists and is unique, and thus the KKT point is the unique solution. Thus, $x^*_{M:N-1}$, $u^*_{M:N-1}$, $\lambda^*_{M:N}$, and $\mu^*_{M:N-1}$ form the solutions of $\mathcal{P}_{M:N}(x^*_M,\lambda^*_N)$.
\end{proof}

Lemma \ref{lem:fixed} justifies the structure of \eqref{eqn:ocp}; specifically, primal and dual information of the neighboring subproblems should be incorporated as initial states and terminal penalties, respectively. Furthermore, it implies that the primal-dual trajectory of the subproblems can be assembled to obtain the optimal primal-dual trajectory of the entire problem if the boundary conditions are set to the optimal values. {From this we also see that a receding-horizon control scheme delivers the solution of the long-horizon problem if the terminal cost gradients are obtained from the dual solution of the long-horizon problem. Recently the effect of terminal cost gradient on the solution trajectory of optimal control problems and its connection with dual variables has been investigated in [17], [18]. In particular, it is shown that economic MPC can achieve asymptotic stability without terminal constraints by incorporating the dual of the underlying steady-state problem as a terminal cost gradient. Furthermore, it is shown that LQ economic MPC can be stabilized by adaptively tuning the terminal cost gradient. Such observations align with our observations of Lemma \ref{lem:fixed}. Specifically, the quality of the solutions of OCPs (which in general improves with the horizon length) can also be improved by choosing a suitable terminal cost gradient.}

We now state our sufficiency condition for convergence, that we call asymptotic decay of sensitivity (ADS). 

\begin{property}[Asymptotic decay of sensitivity]\label{property-weak}
  Assume that \eqref{eqn:ocp} satisfies Assumption \ref{ass:uniqsol}. For given $M$, $N\in\mathbb{Z}_{>0}$ with $M< N$ and $(x^*_M,\lambda^*_N)$, $(\tilde{x}_M,\tilde{\lambda}_N)\in X\times\Lambda$, let $z^*_{M:N}$ and $\tilde{z}_{M:N}$ be the solutions of $\mathcal{P}_{M:N}(x^*_M,\lambda^*_N)$ and $\mathcal{P}_{M:N}(\tilde{x}_M,\tilde{\lambda}_N)$, respectively. There exist $\{\varepsilon_i\}_{i\in\mathbb{Z}_{\geq 0}}$  with $\varepsilon_i \rightarrow 0$ as $i\rightarrow\infty$
  such that:
  \begin{align*}
    \begin{aligned}
    &\Vert z^*_i-\tilde{z}_i\Vert_\infty\leq \varepsilon_{i-M} \Vert x_M^*-\tilde{x}_M\Vert_\infty + \varepsilon_{N-i}\Vert\lambda^*_N - \tilde{\lambda}_N\Vert_\infty
    \end{aligned}
  \end{align*}
 holds for any $i\in\mathbb{I}_{M:N}$.
\end{property}
Property \ref{property-weak} implies that the solution $z_i^*=(x^*_i,\lambda^*_i)$ of $\mathcal{P}_{M:N}(x^*_M,\lambda^*_N)$ at time $i$ becomes less sensitive to perturbations in the initial state $x^*_M$ and the terminal cost gradient $\lambda^*_N$ as the time index $i$ moves away from the boundary.

\begin{remark}
    It is important that the sequence $\{\varepsilon_i\}_{i\in\mathbb{Z}_{>0}}$ is a {\em uniform parameter} that does not depend on $M$ and $N$ and that does not depend on the choice of the boundary conditions (i.e., $x^*_M$ and $\lambda^*_N$). This enables a uniform bound that holds for different subproblems.
\end{remark}
{\begin{remark}[Relation with turnpike properties]
  Property 1 is related but not equivalent to so-called turnpike properties of OCPs (see \cite{epfl:faulwasser15h,faulwasser2018economic}). Property 1 establishes a relationship between two solution trajectories, while a classical turnpike property compares a single solution trajectory with the steady-state solution.  In addition, Property 1 requires asymptotic convergence of the difference between the solutions, while the steady-state turnpike requires a bound on the number of time indexes with $\Vert x_i-x_{\text{s}}\Vert>\epsilon_{\text{turnpike}}$. An in-depth investigation of the relation between Property 1 and time-varying turnpikes is subject to future work.
\end{remark}}

We now prove that the ADS property provides a sufficient condition guaranteeing convergence of the proposed decomposition scheme. 

\begin{theorem}[Convergence]\label{thm:convergence}
  Assume that \eqref{eqn:ocp} satisfies Assumption \ref{ass:uniqsol} and that Property \ref{property-weak} holds. 
  Furthermore, consider scheme \eqref{eqn:alg} with overlap $\omega$  and  let $z^*_{M:N}$ be the solution of $\mathcal{P}_{M:N}(x^*_M,\lambda^*_N)$. We have that:
  \begin{align}\label{eqn:thm-convergence}
    \Vert z^{(t)}_{M:N}-z^*_{M:N}\Vert_\infty \leq \left(2 \varepsilon_\omega \right)^t \Vert z^{(0)}_{M:N}-z^*_{M:N}\Vert_\infty
  \end{align}
\end{theorem}

\begin{proof}
  From Lemma \ref{lem:fixed} we have that the solution of $\mathcal{P}_{M:N}(x^*_M,\lambda^*_N)$ on $\mathbb{I}_{M^\omega_k:N^\omega_k}$ can be obtained from $\mathcal{P}_{M^\omega_k,N^\omega_k}(x^*_{M^\omega_k},\lambda^*_{N^\omega_k})$. For each $k\in\mathbb{I}_{1:K}$, applying Property \ref{property-weak} to $\mathcal{P}_{M^\omega_k,N^\omega_k}(x^*_{M^\omega_k},\lambda^*_{N^\omega_k})$ and $\mathcal{P}_{M^\omega_k,N^\omega_k}(x^{(t)}_{M^\omega_k},\lambda^{(t)}_{N^\omega_k})$ yields
  \begin{align}\label{eqn:pf-1}
    &\Vert z^{(t+1)}_i- z^*_i \Vert_\infty\leq \\
    & \varepsilon_{i-M^\omega_k} \Vert x^{(t)}_{M^\omega_k}-x_{M^\omega_k}^*\Vert_\infty + \varepsilon_{N^\omega_k-i}\Vert\lambda^{(t)}_{N^\omega_k}-\lambda^*_{N^\omega_k}\Vert_\infty\nonumber
  \end{align}
  for any $i\in \mathbb{I}_{M_k:N_k}$ and $t\in\mathbb{Z}_{>0}$. For $i\in \mathbb{I}_{M_1:N_1}$, we have that $x^{(t)}_{1} = x^*_1$ and $i\leq N^\omega_1 - \omega$ yield
  \begin{subequations}\label{eqn:pf-bd}
    \begin{align}\label{eqn:pf-bd-1}
      \Vert z^{(t+1)}_i-z^*_i \Vert_\infty &\leq \varepsilon_\omega \Vert\lambda^{(t)}_{N^\omega_1}-\lambda^*_{N^\omega_1}\Vert_\infty
    \end{align}
    Similarly, for $i\in \mathbb{I}_{M_K:N_K}$, we have that
    \begin{align}\label{eqn:pf-bd-N}
      \Vert z^{(t+1)}_i-z^*_i\Vert_\infty &\leq \varepsilon_\omega \Vert x^{(t)}_{M^\omega_K}-x^*_{M^\omega_K}\Vert_\infty.
    \end{align}
    For $k\in\mathbb{I}_{2:K-1}$, we have $\mathbb{I}_{M_k:N_k}=\mathbb{I}_{M^\omega_k+\omega : N^\omega_k-\omega}$. From \eqref{eqn:pf-1}, we have for any $k\in\mathbb{I}_{2:K-1}$ and $i\in\mathbb{I}_{M_k:N_k}$ that
    \begin{align}\label{eqn:pf-bd-k}
      \left\Vert z^{(t+1)}_i-z^*_i \right\Vert_\infty \leq 2 \varepsilon_{\omega} \left\Vert (x^{(t)}_{M^\omega_k},\lambda^{(t)}_{N^\omega_k})-(x_{M^\omega_k}^*,\lambda^*_{N^\omega_k})\right\Vert_\infty.
    \end{align}
  \end{subequations}
  From \eqref{eqn:pf-bd}, we have that 
  \begin{align}\label{eqn:pf-bd-final}
    \Vert z^{(t+1)}-z^* \Vert_\infty \leq 2 \varepsilon_{\omega} \Vert z^{(t)}_{M:N} -  z^{*}_{M:N} \Vert_\infty
  \end{align}
  Equation \eqref{eqn:pf-bd-final} establishes \eqref{eqn:thm-convergence}.
\end{proof}

Theorem \ref{thm:convergence} indicates that the recursion \eqref{eqn:alg} converges to the solution of the full problem if the size of the overlap $\omega$ is sufficiently large. Furthermore, the convergence rate $2\varepsilon_\omega$ converges asymptotically to zero with the size of the overlap $\omega$ (i.e., with a maximal overlap, the iteration converges in one iteration). This reveals a powerful feature of the proposed scheme: one can {\em control the convergence rate} by choosing the size of the overlap $\omega$. However, as we increase $\omega$, we also increase the complexity of the subproblem and thus a trade-off exists. Accordingly, the selection of a suitable $\omega$ should consider the convergence rate and the subproblem complexity.

We derive termination criteria for the parallel scheme based on the violation of the KKT conditions. We have that \eqref{eqn:kkt-slk} holds at each iteration. We also have that \eqref{eqn:kkt-duf-x}-\eqref{eqn:kkt-prf} hold except for $i\in\{N_1,\cdots,N_{K-1}\}$, which are the boundary indices. For those, the residuals can be evaluated as follows.
\begin{subequations}\label{eqn:res}
  \begin{align}
    \text{residual of \eqref{eqn:kkt-duf-x}} &=(\nabla_{x}f^{(t)}_{N_k})^\top (\lambda^{(t)}_{{N_k}+1} - \hat{\lambda}^{(t,k)}_{{N_k}+1}).\\
    \text{residual of \eqref{eqn:kkt-duf-u}} &=(\nabla_{u}f^{(t)}_{N_k})^\top (\lambda^{(t)}_{{N_k}+1} - \hat{\lambda}^{(t,k)}_{{N_k}+1}).\\
    \text{residual of \eqref{eqn:kkt-prf}}&= x^{(t)}_{N_k+1}-\hat{x}^{(t,k)}_{N_k+1}
  \end{align}
\end{subequations}
Residuals \eqref{eqn:res} can be derived by using the fact that \eqref{eqn:kkt-duf-x}-\eqref{eqn:kkt-prf} hold when $\lambda^{(t+1)}_{i+1}$ and $x^{(t+1)}_{i+1}$ are replaced with $\hat{\lambda}^{(t+1,k)}_{i+1}$ and $\hat{x}^{(t+1,k)}_{i+1}$.
Finally, we define the primal-dual residuals:
\begin{subequations}
  \begin{align*}
    r^{(t)}:=\max_{k\in\mathbb{I}_{1:K-1}}\Vert \hat{x}^{(t,k)}_{N_k+1} - x^{(t)}_{N_k+1} \Vert_\infty\\
    s^{(t)}:=\max_{k\in\mathbb{I}_{1:K-1}}\Vert \hat{\lambda}^{(t,k)}_{N_k+1}- \lambda^{(t)}_{N_k+1}\Vert_\infty,
  \end{align*}
\end{subequations}  
and establish the termination criteria:
\begin{align}
   r^{(t)}<\epsilon_{\text{pr,tol}},\quad s^{(t)}<\epsilon_{\text{du,tol}}. 
\end{align}

\section{Linear Quadratic Problems}\label{sec:LQ}
In this section we show that Property \ref{property-weak} holds for problems with linear and controllable dynamics, convex quadratic objectives, no inequalities, and a single-variable input. Specifically, we make the following assumption. The analysis reveals connections with well-known optimization sensitivity results.

\begin{assumption}\label{ass:2}
  Consider problem \eqref{eqn:ocp} with:
  \begin{enumerate}
  \item $\ell(x,u):= x^\top Q x - f^\top x + r u^2$ with $Q> 0$ and $r > 0$
  \item $f(x,u):= Ax + bu+c$ with $(A,b)$ controllable.
  \item There are no inequality constraints.
  \end{enumerate}
\end{assumption}
Assumption \ref{ass:2} guarantees that Assumption \ref{ass:uniqsol} holds with $X$, $\Lambda=\mathbb{R}^{n_x}$. Now we state the main theorem of this section.

\begin{theorem}\label{thm:lq}
  Property \ref{property-weak} holds for OCPs \eqref{eqn:ocp} satisfying Assumption \ref{ass:2}. 
\end{theorem}
\begin{proof}
Without loss of generality we assume that $M=1$. For any $\mathbb{I}_{1:N}$, $x_1^*$, $\lambda^*_N\in\mathbb{R}^{n_x}$, the solution of $\mathcal{P}_{1:N}(x^*_1,\lambda^*_N)$ exists and is unique and can be obtained from
  \begin{align}\label{eqn:zlambda}
  \begin{bmatrix}
    H & G^\top\\
    G  
  \end{bmatrix}
  \begin{bmatrix}
    w\\
    \lambda
  \end{bmatrix}
  =
  \begin{bmatrix}
    \zeta^*\\
    \xi^*
  \end{bmatrix},
\end{align}
where we define
  \begin{align*}
    H&:=
    \begin{bmatrix}
      Q&&&&&\\
      &r&&&&\\
      &&\ddots&&&\\
      &&&Q&\\
      &&&&r\\
      &&&&&0
    \end{bmatrix},
    w :=
    \begin{bmatrix}
      x_1\\
      u_1\\
      \vdots\\
      x_{N-1}\\
      u_{N-1}\\
      x_{N}
    \end{bmatrix},
    \zeta^*:=
    \begin{bmatrix}
      f\\
      0\\
      \vdots\\
      f\\
      0\\
      \lambda^*_N
    \end{bmatrix}\nonumber
    \\
    G&:=
    \begin{bmatrix}
      I\\
      -A&-b&I\\
      &&\ddots&\ddots\\
      &&&-A&-b&I
    \end{bmatrix}
    ,
    \xi^*:=
    \begin{bmatrix}
      x^*_1\\
      0\\
      \vdots\\
      0
    \end{bmatrix}.
  \end{align*}
  
  Since $(A,b)$ is controllable, 
  the controllability matrix $\mathcal{C}:=\begin{bmatrix}b &Ab &\cdots&A^{n_x-1}b\end{bmatrix}$ has full row rank.
  This implies that $\mathcal{C} u + A^{n_x} b = 0$ has a unique solution.
  Using this fact, we can construct $\overline{x}_1,\cdots,\overline{x}_{n_x}\in\mathbb{R}^{n_x}$ and $\overline{u}_1,\cdots,\overline{u}_{n_x}\in\mathbb{R}$ with: $\overline{x}_{i+1} = A \overline{x}_i + b\overline{u}_i$ for any $i\in\mathbb{I}_{0:n_x}$, $\overline{x}_0 = 0$, $\overline{u}_0=1$, and $\overline{x}_{n_x+1}=0$. We construct $Z\in\mathbb{R}^{(Nn_x+N-1)\times (N-1)}$ as:
\begin{align}\label{eqn:Z}
  Z:=
  \begin{bmatrix}
    \\
    1\\
    \overline{x}_1\\
    \overline{u}_1&1\\
    \vdots&\vdots\\
    \overline{x}_{n_x}&\overline{x}_{n_x-1}&\ddots\\
    \overline{u}_{n_x}&\overline{u}_{n_x-1}&\cdots&1\\
    &\overline{x}_{n_x}&\cdots&\overline{x}_1&\\
    &\overline{u}_{n_x}&\cdots&\overline{u}_1&1\\
    &&\ddots&\vdots&\vdots\\
    &&&\overline{x}_{n_x}&\overline{x}_{n_x-1}&\ddots\\
    &&&\overline{u}_{n_x}&\overline{u}_{n_x-1}&\cdots&1\\
    &&&&\overline{x}_{n_x}&\cdots&\overline{x}_1
  \end{bmatrix}
\end{align}
One can show that $GZ = 0$ and that $Z$ has full column rank by using the lower-triangular structure of $Z$. Observe that $G\in\mathbb{R}^{Nn_x\times(Nn_x+N-1)}$ has full row rank. By the fundamental theorem of linear algebra, the null space of $G$ has dimension $N-1$; consequently, columns of $Z$ span the null space of $G$. Similarly,
there exists a unique solution of $\mathcal{C}u  +  A^{n_x}e_i = 0$  for any $e_i$ where $e_i$ is the $i$th standard unit vector of $\mathbb{R}^{n_x}$. Using this observation, 
we construct $X_1,\cdots, X_{n_x}\in\mathbb{R}^{n_x\times n_x}$ and $U_1,\cdots, U_{n_x}\in\mathbb{R}^{1\times n_x}$ with $X_1=I$, $X_{n_x+1}=0$, and $X_{i+1} = AX_i + bU_i$ for any $i\in\mathbb{I}_{1:n_x}$. Now consider
\begin{align}
  Y:=
  \begin{bmatrix}
    I\\
    U_1\\
    X_2&I\\
    \vdots&\vdots\\
    U_{n_x}&U_{n_x-1}&\ddots\\
    &X_{n_x}&\cdots&I\\
    &U_{n_x}&\cdots&U_1&\\
    &&&X_1&I\\
    &&\ddots&\vdots&\vdots\\
    &&&U_{n_x}&U_{n_x-1}&\ddots\\
    &&&&X_{n_x}&\cdots&I
  \end{bmatrix}
\end{align}
and observe that $GY = I$ holds. We now apply a null-space projection to \eqref{eqn:zlambda} using $Z$ and $Y$. This way we obtain the equivalent unconstrained QP
\begin{align}\label{eqn:unconstrainedqp}
  \min_{p\in\mathbb{R}^{N-1}} (Zp+Y\xi^*)^\top H (Z p +Y\xi^*) - (Z^\top \zeta^*)^\top p.
\end{align}
 The solution $p^*$ of \eqref{eqn:unconstrainedqp} is given by
\begin{align}\label{eqn:pstar}
  p^* = \left(Z^\top H Z\right)^{-1} \left(Z^\top \zeta^* + Z^\top H Y \xi^*\right).
\end{align}
We have that $w^*= Zp^*+Y\xi^*$ and $ \lambda^* = -Y^\top H w^* + Y^\top \zeta^*$. Using \eqref{eqn:pstar}, we can write
\begin{subequations}\label{eqn:clsol}
  \begin{align}
    w^* =&  \left(Z\overline{H} Z^\top \right)\zeta^* + \left(Z\overline{H} W^\top \right) \xi^*\\
    \lambda^* =& \left(-W \overline{H} Z^\top + Y^\top \right)\zeta^* + \left(W\overline{H} W^\top\right) \xi^*,
  \end{align}
\end{subequations}
where $\overline{H}:=(Z^\top H Z)^{-1}$ and $W:=Y^\top HZ$. Since $Z$ and $Y$ are independent of the choice of $(x^*_1,\lambda^*_N)$, we obtain the solution of the form \eqref{eqn:clsol} with different boundary conditions $(\tilde{x}_1,\tilde{\lambda}_N)$. Thus, we have that
\begin{subequations}\label{eqn:difference}
\begin{align}
     w^* - \tilde{w} =& \left( Z\overline{H} Z^\top \right )\Delta \zeta + \left( Z\overline{H}W^\top \right)  \Delta \xi\\
     \lambda^* - \tilde{\lambda} =& \left(-W\overline{H} Z^\top + Y^\top \right)\Delta \zeta+ \left(W\overline{H} W^\top \right) \Delta \xi
\end{align}
\end{subequations}
where $\Delta\zeta:=\zeta^*-\tilde{\zeta}$ and $\Delta\xi:=\xi^* - \tilde{\xi}$. 
Considering
\begin{subequations}\label{eqn:bound2}
  \begin{align}
    L_H &:= \left\Vert
    \begin{bmatrix}
      Q&\\
      &r
    \end{bmatrix}
    \right\Vert_\infty
    \\
    L_Z&:=\left\Vert
    \begin{bmatrix}
      \overline{x}_{n_x-1} & \overline{x}_{n_x-2} & \cdots & \overline{x}_1 \\
      \overline{u}_{n_x-1} & \overline{u}_{n_x-2} & \cdots & \overline{u}_1 & 1\\
    \end{bmatrix}
    \right\Vert_\infty
    \\
    L_{Z^\top} &:=\left\Vert
    \begin{bmatrix}
      1 &\overline{x}^\top_1 &\overline{u}_1^\top &\cdots & \overline{x}^\top_{n_x} & \overline{u}^\top_{n_x}
    \end{bmatrix}
    \right\Vert_\infty
    \\
    L_Y&:=
    \left\Vert
    \begin{bmatrix}
      X_{n_x} & \cdots & I\\
      U_{n_x} & \cdots & U_1\\
    \end{bmatrix}
    \right\Vert_\infty\\
    L_{Y^\top}&:=
    \left\Vert
    \begin{bmatrix}
      I& U_1^\top & X_1^\top & \cdots & X_{n_x}^\top & U_{n_x}^\top
    \end{bmatrix}
    \right\Vert_\infty\\
    L_{W} &:=L_{Y^\top}L_H L_Z,\quad L_{W^\top} :=L_{Z^\top}L_H L_{Y^\top}\\
    L&:=\max\{L_{Z},L_{Z^\top},L_{W},L_{W^\top}\}
  \end{align}
\end{subequations}
and we can see that $\Vert Z \Vert_\infty\leq L_Z$, $\Vert Z^\top \Vert_\infty\leq L_{Z^\top}$, $\Vert Y \Vert_\infty\leq L_Y$, $\Vert Y^\top \Vert_\infty\leq L_{Y^\top}$, $\Vert H \Vert_\infty \leq L_H$, $\Vert W\Vert \leq L_W$, and $\Vert W^\top\Vert \leq L_{W^\top}$. Quantities defined in \eqref{eqn:bound2} are all uniform (i.e., does not depend on the length of problem and boundary conditions).

The bandwidth $\mathcal{B}(\cdot)$ of a matrix is defined as the smallest integer such that $|i-j|\leq \mathcal{B}(\cdot)$ for any $(\cdot)_{i,j}\neq 0$. Note that the bandwidth of $Z^\top HZ$ is not greater than $n_x$. The following is a modification of \cite[Corollary 1]{shin2018decentralized}.
\begin{proposition}\label{cor:domdec}
  Consider positive definite $\Gamma \in\mathbb{R}^{n\times n}$. Suppose that $\lambda(\Gamma)\in[\lambda_{\min},\lambda_{\max}]$ for some $\lambda_{\min},\lambda_{\max}\in\mathbb{R}_{>0}$. Then the following holds.
  \begin{align*}
    \left|(\Gamma^{-1})_{i,j}\right| &\leq \frac{1}{\lambda_{\min}} \left(\frac{\lambda_{\max}-\lambda_{\min}}{\lambda_{\max}+\lambda_{\min}}\right)^{\frac{|i-j|}{\mathcal{B}(\Gamma)}}\;\forall i,j\in\mathbb{I}_{1:n}
  \end{align*}
\end{proposition}

\begin{proof}
  The proof is given in \cite{shin2018decentralized}.
\end{proof}

The following lemma establishes that there exist uniform upper and lower bounds for the eigenvalues of $Z^\top HZ$.

\begin{lemma}\label{lem:mineig}
  Consider \eqref{eqn:ocp}, such that Assumption \ref{ass:2} and let $Z$ be from \eqref{eqn:Z}. Then there exist uniform parameters $\lambda_{1},\lambda_{2}\in\mathbb{R}_{>0}$ (independent of the choice $N$) such the eigenvalues $\lambda$ of $Z^\top H Z$ satisfy $\lambda\in[\lambda_{1},\lambda_{2}]$.
\end{lemma}
\begin{proof}
  The upper bound comes from
  \begin{align*}
    \lambda(Z^\top H Z) &\leq \Vert Z^\top  H Z \Vert_\infty \leq L^2 L_H
  \end{align*}
 and thus we define $\lambda_2:=L^2 L_H$. We can find $\lambda_1$ from the lower bound of $p^\top (Z^\top H Z) p$ for $p\in\mathbb{R}^{N-1}$ with $\Vert p \Vert_2 = 1$. We have that
  \begin{align}\label{eqn:H0}
    &p^\top (Z^\top H Z) p =
    \sum_{i=1}^{N-1}
    p_{i-N+1:i}
    \hat{H}
    p_{i-N+1:i}
  \end{align}
  where $p_i=0$ for $i\leq 0$ for convenience and
    \begin{align*}
      \hat{H}&:=
      \hat{Z}^\top
      \begin{bmatrix}
        Q & 0\\
        0 & r
      \end{bmatrix}\hat{Z},\;
      \hat{Z}:=
      \begin{bmatrix}
        \overline{x}_{n_x} & \cdots & \overline{x}_1 &\\
        \overline{u}_{n_x} & \cdots & \overline{u}_1 & 1
      \end{bmatrix}.
    \end{align*}

    Now we show that $\overline{x}_1,\cdots,\overline{x}_{n_x}$ are linearly independent. To establish a contradiction, suppose that $\overline{x}_1,\cdots,\overline{x}_{n_x}$ are linearly dependent; then there exist not all-zero $\alpha_1,\cdots,\alpha_{n_x}$ such that $\alpha_1 \overline{x}_1 + \cdots + \alpha_{n_x} \overline{x}_{n_x} = 0$ holds. This implies that
  \begin{align*}
    \alpha_{1} \left(b u_0\right) + \cdots +\alpha_{n_x} \left(\sum_{i=0}^{n_x-1} A^{n_x-1-i} b u_i\right)=0.
  \end{align*}
  By rearranging, we obtain
  \begin{align}\label{eqn:ld-2}
    \left(\alpha_{1}u_0 + \cdots + \alpha_{n_x}u_{n_x-1}  \right)b + \cdots +\left(\alpha_{n_x}u_{0}\right) A^{n_x-1}b = 0.
  \end{align}
  Since $u_0=1$ and $\alpha_1,\cdots,\alpha_{n_x}$ are not all-zero, a non-zero coefficient exists in \eqref{eqn:ld-2}, and thus $b,\cdots,A^{n_x-1}b$ are linearly dependent. This contradicts the assumption that $(A,b)$ is controllable.  Consequently, $\overline{x}_1,\cdots,\overline{x}_{n_x}$ are linearly independent.

  From the linear independence of $\overline{x}_1,\cdots,\overline{x}_{n_x}$, one can show that $\hat{Z}$ has full column rank and thus $\hat{H}$ is positive definite. From \eqref{eqn:H0} one can show that
  \begin{subequations}
  \begin{align*}
    p^\top (Z^\top H Z) p
    &\geq
    \sum_{i=1}^{N-1}
    \lambda_{\min}(\hat{H})
    \Vert p_{i-N+1:i}\Vert_2^2\\
    &\geq \lambda_{\min}(\hat{H}) \Vert p \Vert_2^2 = \lambda_{\min}(\hat{H}).
  \end{align*}
  \end{subequations}
  Observe that $\hat{H}$ is independent of the length of the problem. Thus, we can set uniform parameter $\lambda_{1} := \lambda_{\min}(\hat{H})$.
\end{proof}

By Proposition \ref{cor:domdec} and Lemma \ref{lem:mineig}, we have that
\begin{align}\label{eqn:spmatbound}
  \left\vert\left(\overline{H}\right)_{i,j}\right\vert \leq  \frac{1}{\lambda_{1}}\left(\frac{\lambda_{2}-\lambda_{1}}{\lambda_{2}+\lambda_{1}}\right)^{|i-j|/n_x},\; \forall i,j\in\mathbb{I}_{1:N-1}
\end{align}
where $\lambda_{1}$ and $\lambda_{2}$ are the uniform upper and lower bound of the eigenvalues of $Z^\top H Z$ established in Lemma \ref{lem:mineig}.

Finally, inspecting \eqref{eqn:difference} we observe that $Z$ and $W$ are sparse. In particular, there exists uniform parameter $N_s$ such that the following holds for $|j-i|\geq N_s$.
\begin{align}\label{eqn:sparse}
  (Z)_{\alpha(i),j}=0,\quad (W)_{\beta(i),j}=0  
\end{align}
where $\alpha(i) := \mathbb{I}_{(i-1)(n_x+1)+1:in_x+i-1}$ and $\beta(i) := \mathbb{I}_{(i-1)n_x+1:in_x}$ (these index sets correspond to the index of $x_i$ and $\lambda_i$ among the entries of $w$ and $\lambda$, respectively). Using the sparsity structure identified in \eqref{eqn:sparse}, one can show the following quantities are less than or equal to $L^2\Vert(\overline{H})_{i\pm N_s,j\pm N_s} \Vert_\infty$
\begin{subequations}\label{eqn:bound3}
  \begin{align}
    &\Vert (Z \overline{H}Z^\top)_{\alpha(i),\alpha(j)} \Vert_\infty,
    \Vert (W \overline{H}Z^\top)_{\beta(i),\alpha(j)} \Vert_\infty\\
    &\Vert (Z \overline{H}W^\top)_{\alpha(i),\beta(j)} \Vert_\infty,
    \Vert (W \overline{H} W^\top )_{\beta(i),\beta(j)}\Vert_\infty
  \end{align}
    \end{subequations}
  where we use the syntax $i\pm N_s := \mathbb{I}_{i-N_s:i+N_s}$ and define $(\overline{H})_{i,j}:=0$ if $\{i,j\}\not\subseteq \mathbb{I}_{1:N-1}$ for convenience.
  From \eqref{eqn:difference} and \eqref{eqn:bound3}, we have
  \begin{subequations}
    \begin{align*}
      &\Vert x^*_i- \tilde{x}_i\Vert_\infty \leq L^2 \left\Vert(\overline{H})_{i\pm N_s,N-1\pm N_s}\right\Vert_\infty \Vert \Delta  \lambda_N\nonumber\Vert_\infty\\
      & \qquad+ L^2\left\Vert(\overline{H})_{i\pm N_s,1\pm N_s}\right\Vert_\infty \Vert \Delta x_1 \Vert_\infty\\
  &\Vert \lambda^*_i- \tilde{\lambda}_i\Vert_\infty \leq L^2  \left\Vert(\overline{H})_{i\pm N_s,1\pm N_s}\right\Vert_\infty \Vert\Delta x_1\nonumber\Vert_\infty\\
  &+ \big(L^2 \left\Vert(\overline{H})_{i\pm N_s,N-1\pm N_s}\right\Vert_\infty + \mathbf{1}_{i=N} \big)\Vert \Delta \lambda_N\Vert_\infty.
    \end{align*}
  \end{subequations}
  By \eqref{eqn:spmatbound} and $\mathbf{1}_{i=N}\leq \rho^{N-i-2N_s-1}$, we have that
\begin{subequations}\label{eqn:xlbound}
\begin{align}\label{eqn:xbound}
  &\Vert x^*_i- \tilde{x}_i\Vert_\infty \leq L^2 \frac{2N_s+1}{\lambda_1}\rho^{N-i-2N_s-1} \Vert \Delta \lambda_N\Vert_\infty\nonumber
  \\
  &\quad + L^2 \frac{2N_s+1}{\lambda_1}\rho^{i-2N_s-1} \Vert\Delta x_1\Vert_\infty\\
\label{eqn:lbound}
  &\Vert \lambda^*_i- \tilde{\lambda}_i\Vert_\infty \leq L^2  \frac{2N_s+1}{\lambda_1}\rho^{i-2N_s-1}\Vert \Delta x_1\Vert_\infty \nonumber \\
  & + \big(L^2 \frac{2N_s+1}{\lambda_1}+ 1 \big)\rho^{N-i-2N_s-1}\Vert \Delta \lambda_N\Vert_\infty 
\end{align}
\end{subequations}
for $i\in \mathbb{I}_{M:N}$ and $\rho:=(\lambda_{2}-\lambda_{1})/(\lambda_{2}+\lambda_{1})$. We define
  \begin{align}\label{eqn:epsilon}
    \varepsilon_i &:=\left(L^2 \frac{2N_s+1}{\lambda_1}+1\right) \rho^{i-2N_s-1},\; \forall i\in\mathbb{Z}_{>0}.
  \end{align}
Note that $\{\varepsilon_i\}_{i\in\mathbb{Z}_{\geq 0}}$ is a uniform parameter and \eqref{eqn:xlbound}-\eqref{eqn:epsilon} establish Property \eqref{property-weak}. This concludes the proof.  
\end{proof}

Theorem \ref{thm:lq}  implies that the reduced Hessian is positive definite--a key requirement in optimization sensitivity results as the solution must be locally unique and bounded perturbations yield bounded differences in solutions. We thus expect that Theorem \ref{thm:lq} can be generalized to LQ OCPs with multiple inputs, time-variant objectives and dynamics, and inequality constraints by exploiting algebraic properties. Here, we focus on a simple setting due to space limitations and to keep the presentation clear. For an even more general setting (nonlinear, inequality-constrained) setting, one can seek to validate the sensitivity property numerically by using simulations. We show how to do this in the next section.

\section{Numerical Examples}\label{sec:cstudy}
We use a nonlinear OCP to illustrate that Property \ref{property-weak} guarantees convergence of the decomposition algorithm and to highlight that the approach achieves faster solutions than off-the-shelf solvers. {We consider a nonlinear economic MPC problem for a chemical reactor \cite{faulwasser2018economic,bailey1971cyclic}. The dynamics of the reactor are given by:}
\begin{subequations}\label{eqn:cstr}
  \begin{align}
    \frac{dc_A}{dt} &= 1-10^4(c_A)^2 e^{-\frac{1}{T}} + 400c_A e^{-\frac{0.55}{T}} - c_A\\
    \frac{dc_B}{dt} &= 10^4(c_A)^2 e^{-\frac{1}{T}} - c_B \\
    \frac{dT}{dt} &= u-T
  \end{align}
Here, $c_A$, $c_B$, and $T$ are the concentration of $A$, the concentration of $B$, and the temperature; $x:=(c_A,c_B,T)$ are the state variables; $u$ is the input variable. The objective is a combination of an economic objective (maximizing the production of $B$) and a convex regularization term:
\begin{align}
  \ell(x,u):= -c_B + \rho_{\text{reg}}( u-u_s)^2
\end{align}
where $u_s$ is determined by solving the underlying steady-state optimization problem that minimizes $-c_B$ (see \cite{epfl:faulwasser15h}). The following inequality constraints are enforced on the states and the inputs:
\begin{align}
  c_A,c_B,T\geq 0\quad\text{and}\quad 0.049 \leq u\leq 0.449
\end{align}
\end{subequations}
The control step length is $1$ sec and the differential equations are discretized using an implicit Euler scheme with a step length of $0.25$ sec. We implemented a parallel version of the scheme in {\tt Julia}. Problems were formulated in the modeling language {\tt JuMP} \cite{dunning2017jump}, and were solved with the nonlinear programming solver {\tt IPOPT} \cite{wachter2006implementation}. The scheme was executed on an Intel Xeon CPU E5-2698 v3 processor running at 2.30GHz.

We use numerical simulations to verify that Property \ref{property-weak} holds.
Here, we assess the sensitivity of the primal-dual trajectories against perturbations to the initial state and the terminal cost gradient. The reference problem is formulated with $N=600$ (i.e., $10$ mins) and boundary conditions with $(x^*_1,\lambda^*_N)$. The boundary conditions are perturbed as $(x^*_1,\lambda^*_N)+\delta$, where $\delta$ is sampled from a normal random variable.
It is known that the OCP with \eqref{eqn:cstr} has a periodic optimal solution when $\rho_{\text{reg}}=0$ and a steady-state solution with $\rho_{\text{reg}}=0.5$ \cite{faulwasser2018economic}. The solutions of the reference (unperturbed) problem and 30 samples of perturbed problems are shown in Figure \ref{fig:sens-dis} ($\rho_{\text{reg}}=0.5$) and Figure \ref{fig:sens-non} ($\rho_{\text{reg}}=0$). When $\rho_{\text{reg}}=0.5$ we see that, for both primal and dual solutions, the distances from the reference trajectories are small in the center of the time domain and grow as we approach the boundary. {On the other hand, when $\rho_{\text{reg}}=0$, such convergence is not observed. This indicates that Property \ref{property-weak} holds for $\rho_{\text{reg}}=0.5$ but does not hold for $\rho_{\text{reg}}=0$.} {\cmred This implies that the ADS property strongly depends on the stage cost $\ell$. For high-dimensional systems where it is difficult to graphically assess the ADS property, one may consider assessing the error trajectory based on the norm of deviation from the reference trajectory.}

\begin{figure*}[t!]\centering
  \includegraphics[width=.49\textwidth]{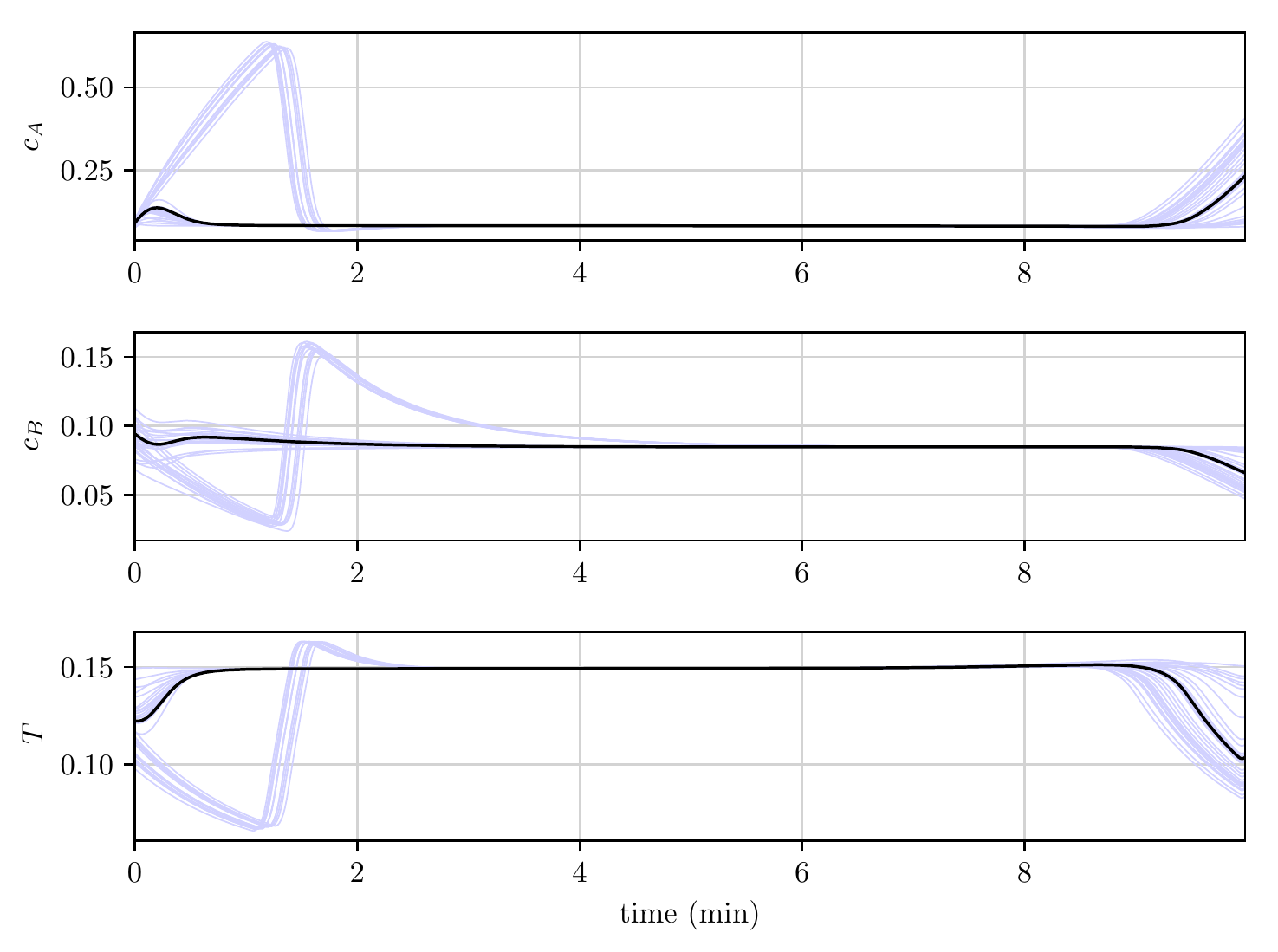}
  \includegraphics[width=.49\textwidth]{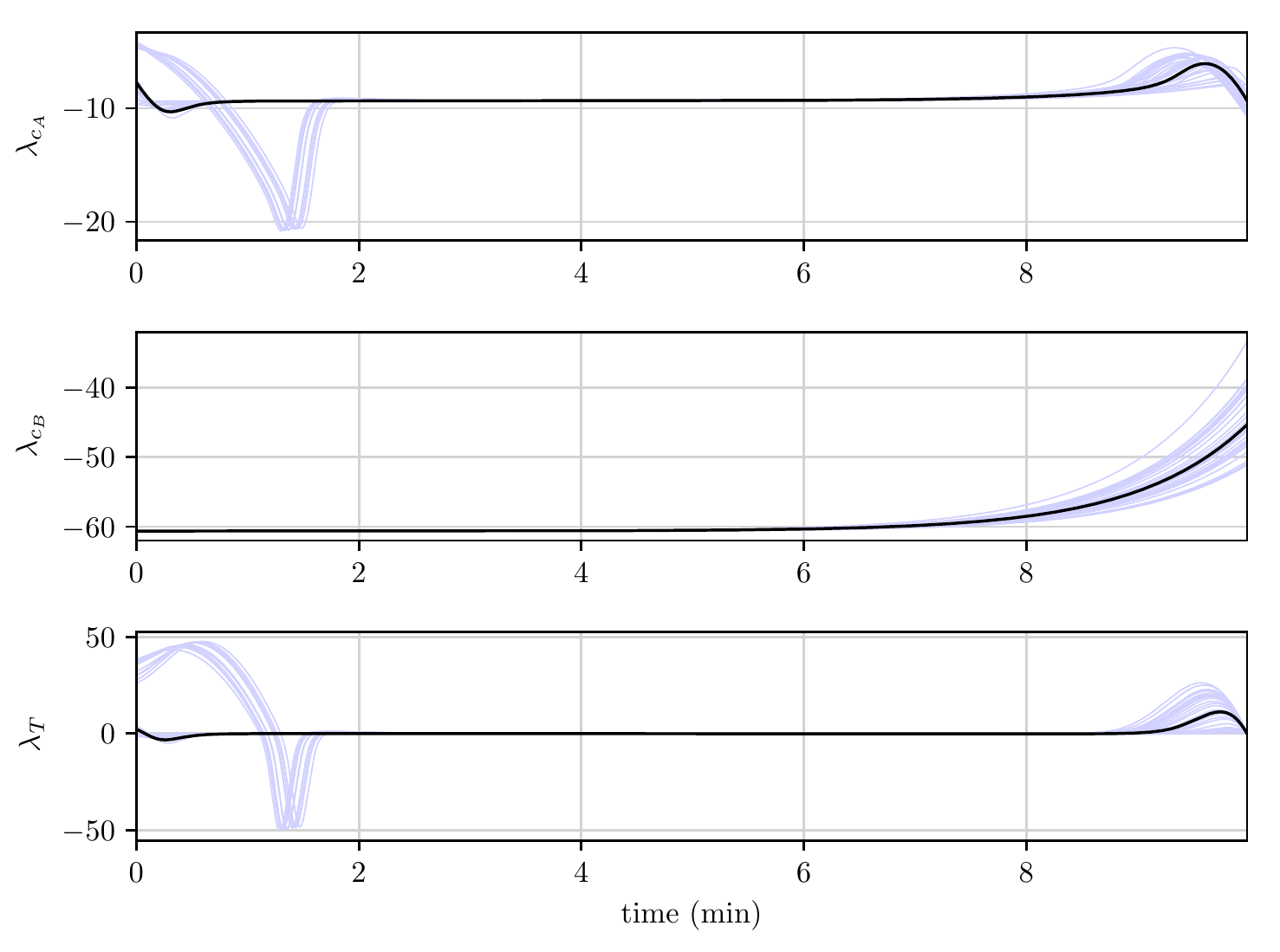}
  \caption{The primal and dual trajectories of the reference problem (black) and 30 perturbed problems (light blue) with $N=180$ and $\rho_{\text{reg}}=0.5$.}\label{fig:sens-dis}
\end{figure*}
\begin{figure*}[t!]\centering
  \includegraphics[width=.49\textwidth]{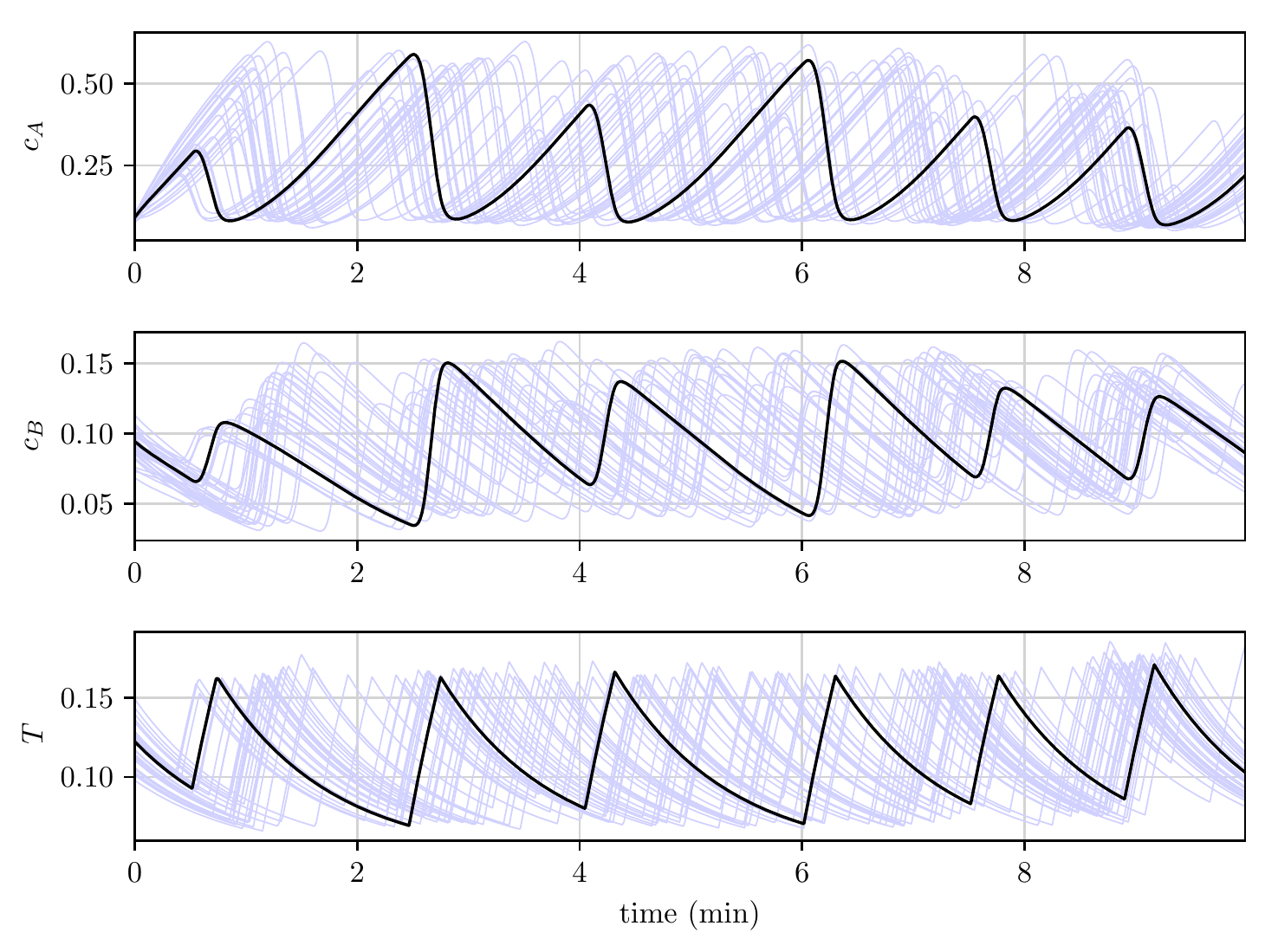}
  \includegraphics[width=.49\textwidth]{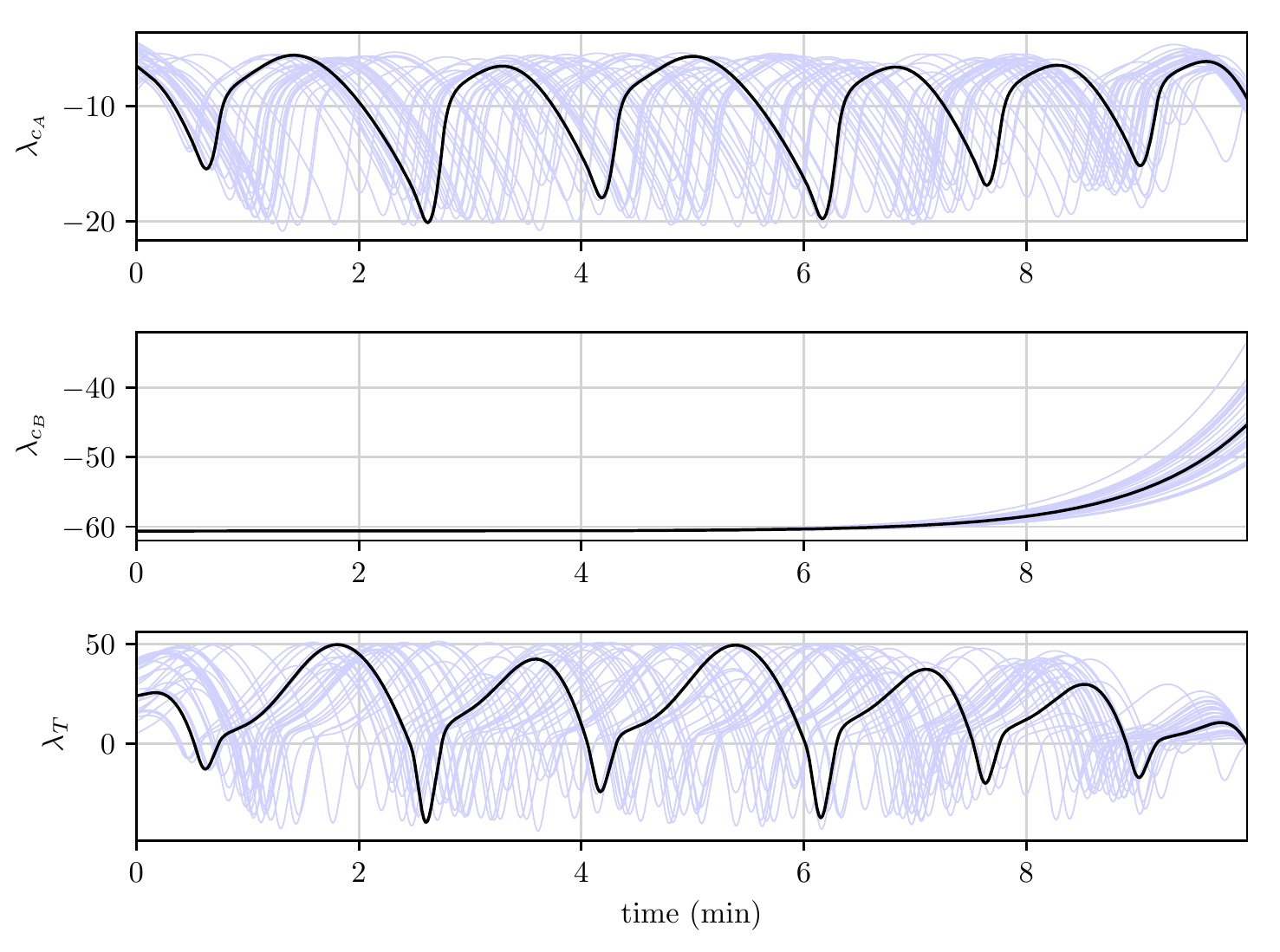}
  \caption{The primal and dual trajectories of the reference problem (black) and 30 perturbed problems (light blue) with $N=180$ and $\rho_{\text{reg}}=0$.}\label{fig:sens-non}
\end{figure*}

We next verify that the decomposition scheme converges and that the convergence rate improves with the size of the overlap. A problem with $N=4,800$ (i.e., $80$ mins) was decomposed into $8$ subdomains and solved with $\omega=2$, $4$, and $8$ and $\rho_{\text{reg}}=0.5$. Figure \ref{fig:convergence} (left) shows the evolution of the residuals and confirms that the rate improves with the overlap. {The algorithm did not converge when $\rho_{\text{reg}}=0$. These observations reinforce the key role of Property \ref{property-weak}.}

We also explore the computational efficiency achieved via parallel decomposition. A problem with $N=4,800$ was solved using {\tt IPOPT}  and we compared the solution time against that of the decomposition scheme with an overlap of $\omega=8$ and $\rho_{\text{reg}}=0.5$. The solution found by the decomposition scheme was equal to the solution from {\tt IPOPT}. Figure \ref{fig:convergence} (right) clearly illustrates that the solution time decreases as the number of computing cores increases (this also increases the number of subdomains). We can see that, if a sufficient number of cores are used, the proposed scheme becomes faster than {\tt IPOPT}. 

\begin{figure*}[t!]\centering
  \includegraphics[width=.49\textwidth]{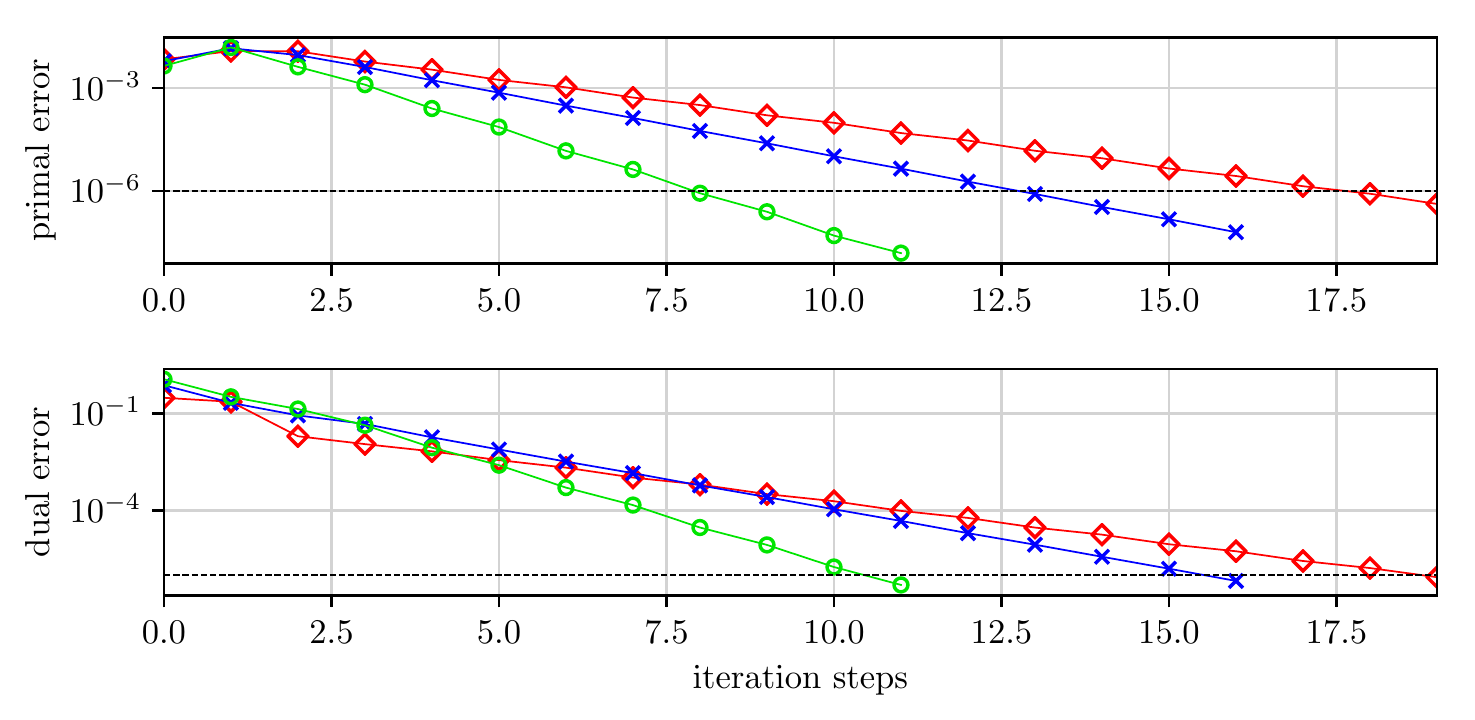}
  \includegraphics[width=.49\textwidth]{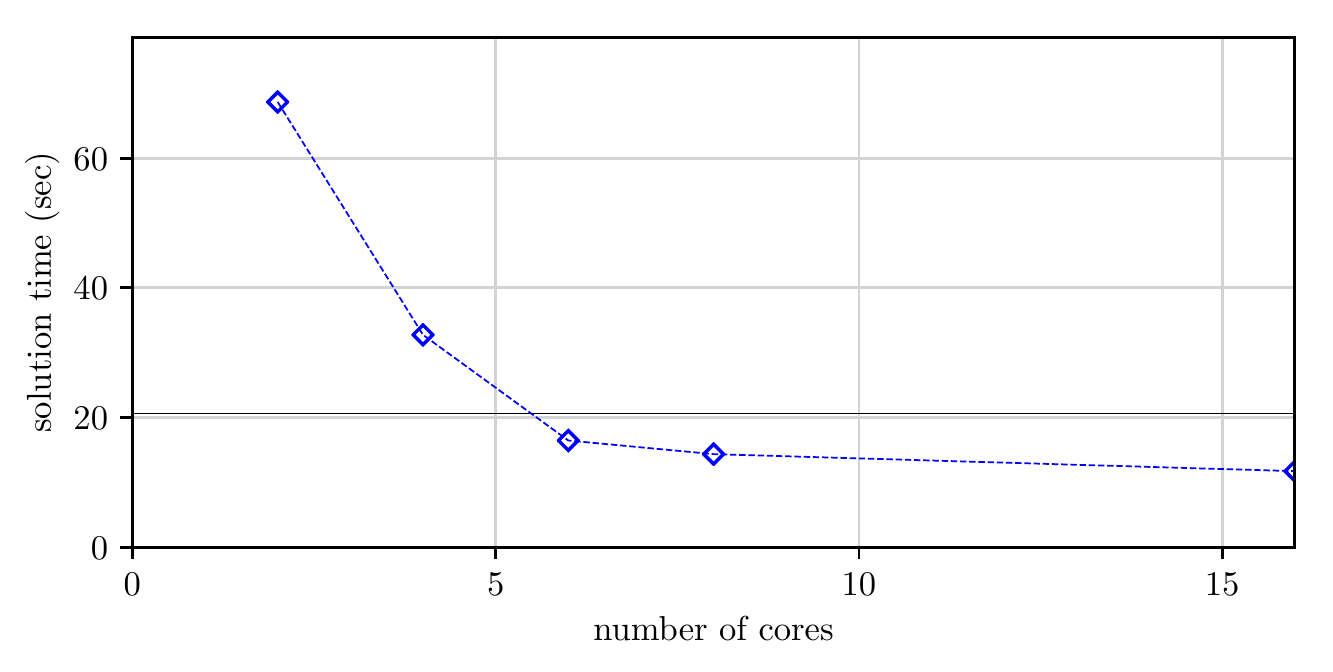}
  \caption{Left: the evolution of the primal and dual residuals with $\omega=2$ (red and diamond), $\omega=4$ (blue and cross), and $\omega=8$ (green and circle) and with $N=4,800$, $K=8$, and $\rho_{\text{reg}}=0.5$. Right: a performance comparison of the centralized solver (black) and the proposed scheme with different number of cores(blue and diamond) and with  $N=4,800$, $\omega=8$, and $\rho_{\text{reg}}=0.5$.}\label{fig:convergence}
\end{figure*}

\section{Conclusions}\label{sec:conclusion}
We have presented a temporal decomposition scheme for long-horizon OCPs that solves subproblems on overlapping time domains. We have proposed an asymptotic sensitivity decay property that guarantees convergence of the algorithm. This property indicates that the sensitivity of the primal and dual trajectories to perturbations on the boundary conditions (initial state and terminal cost gradient) decays asymptotically as one move away from the boundaries. This property also indicates that the scheme converges if the overlap is sufficiently large and the convergence rate improves with the size of the overlap. {\cm We have demonstrated that the solution time of long-horizon OCPs can be improved with the proposed decomposition method.}
\bibliographystyle{ieeetran}
\bibliography{dissipativity}
\end{document}